\documentclass[12pt]{amsart}
\usepackage{amsfonts,amssymb,verbatim,amsthm,graphicx}
\usepackage[active]{srcltx}

\newtheorem{thm}{Theorem}[section]

\newtheorem{prop}[thm]{Proposition}
\theoremstyle{definition}

\theoremstyle{remark}

\numberwithin{equation}{section}

\newcommand{\R}{\mathbb R}

\def\Lc{\mathcal{L}}

\def\d{\partial}

\def\0{\varnothing}

\def\le{\leqslant}
\def\ge{\geqslant}
\newcommand\pd[2]{\frac{\partial #1}{\partial #2}}

\def\R{\mathbb{R}}

\def\0{\varnothing}

\begin{document}

\title[Number of vertices in Gelfand--Zetlin polytopes]
{Number of vertices in Gelfand--Zetlin polytopes}
\author{Pavel Gusev}
\author{Valentina Kiritchenko}
\author{Vladlen Timorin}

\address[PG, VK and VT]{Faculty of Mathematics and Laboratory of Algebraic Geometry\\
National Research University Higher School of Economics\\
7 Vavilova St 117312 Moscow, Russia}

\address[VK]
{RAS Institute for Information Transmission Problems\\
Bolshoy Karetny Pereulok 19, 127994 Moscow, Russia}

\address[VT]
{Independent University of Moscow\\
Bolshoy Vlasyevskiy Pereulok 11, 119002 Moscow, Russia}

\email{vtimorin@hse.ru}

\begin{abstract}
We discuss the problem of counting vertices in Gelfand--Zetlin polytopes.
Namely, we deduce a partial differential equation with constant coefficients
on the exponential generating function for these numbers.
For some particular classes of Gelfand-Zetlin polytopes, the number
of vertices can be given by explicit formulas.
\end{abstract}
\maketitle

\section{Introduction and statement of results}

Gelfand--Zetlin polytopes play an important role in representation theory \cite{GZ,O}
and in algebraic geometry (see \cite{KST}).
Let $\lambda_1\le\dots\le\lambda_n$ be a non-decreasing finite sequence
of integers, i.e. an integer partition.
The corresponding Gelfand--Zetlin polytope is a convex polytope
in $\R^{\frac{n(n-1)}2}$ defined by an explicit set of linear inequalities
depending on $\lambda_i$.
It will be convenient to label the coordinates $u_{i,j}$ in $\R^{\frac{n(n-1)}2}$
by pairs of integers $(i,j)$, where $i$ runs from $1$ to $n-1$, and
$j$ runs from $1$ to $n-i$.
The inequalities defining the Gelfand--Zetlin polytope can be visualized
by the following triangular table.
$$
\begin{array}{ccccccccc}
\lambda_1&       & \lambda_2    &         &\lambda_3 & &\ldots & &\lambda_n   \\
    &u_{1,1}&         &u_{1,2}  &         & \ldots  &         &u_{1,n-1}&       \\
    &       &u_{2,1} &         &\ldots   &         &u_{2,n-2}&         &       \\
    &       &         & \ddots  &\ldots   &   &         &         &       \\
    &       &         &u_{n-2,1}&         &u_{n-2,2}&         &         &       \\
    &       &         &         &u_{n-1,1}&         &         &         &       \\
\end{array}
\eqno{(GZ)}
$$
where every triple of numbers $a$, $b$, $c$ that appear in the table as
vertices of the triangle
 $$
 \begin{array}{ccc}
  a &  &b \\
   & c &
 \end{array}
 $$
are subject to the inequalities $a\le c\le b$.

In this paper, we discuss generating functions for the number
of vertices in Gelfand--Zetlin polytopes. We will use the
multiplicative notation for partitions, e.g. $1^{i_1} 2^{i_2} 3^{i_3}$
will denote the partition consisting of $i_1$ copies of $1$,
$i_2$ copies of $2$, and $i_3$ copies of $3$.
Given a partition $p$, we write $GZ(p)$ for the corresponding
Gelfand--Zetlin polytope, and $V(p)$ for the number of vertices in $GZ(p)$.

Fix a positive integer $k$, and consider all partitions of the form
$1^{i_1}\dots k^{i_k}$, where a priori some of the powers $i_j$ may be
zero.
We let $E_k$ denote the exponential generating function for the
numbers $V(1^{i_1}\dots k^{i_k})$, i.e. the formal power series
$$
E_k=\sum_{i_1,\dots,i_k\ge 0}V(1^{i_1}\dots k^{i_k})\frac{z_1^{i_1}}{i_1!}
\dots\frac{z_k^{i_k}}{i_k!}.
$$
Our first result is a partial differential equation on the function $E_k$:

\begin{thm}
\label{t:pde}
  The formal power series $E_k$ satisfies the following partial differential
  equation with constant coefficients:
  $$
  \left(\frac{\d^k}{\d z_1\dots\d z_k}-
  \left(\pd{}{z_1}+\pd{}{z_2}\right)\dots\left(\pd{}{z_{k-1}}+\pd{}{z_k}\right)\right)
  E_k=0.
  $$
\end{thm}

E.g. we have
$$
E_1(z_1)=e^{z_1},\quad E_2(z_1,z_2)=e^{z_1+z_2} I_0\left(2 \sqrt{z_1 z_2}\right),
$$
where $I_0$ is the modified Bessel function of the first kind with parameter $0$.
This function can be defined e.g. by its power expansion
$$
I_0(t)=\sum_{n=0}^\infty \frac{t^n}{n!^2}.
$$

It is also useful to consider ordinary generating functions for
the numbers $V(1^{i_1}\dots k^{i_k})$:
$$
G_k(y_1,\dots,y_k)=\sum_{i_1,\dots,i_k\ge 0}V(1^{i_1}\dots k^{i_k})
y_1^{i_1}\dots y_k^{i_k}.
$$
We will also deduce equations on $G_k$.
These will be difference equations rather than differential equations.
For any power series $f$ in the variables $y_1$, $\dots$, $y_k$,
define the action of the divided difference operator $\Delta_i$ on $f$ as
$$
\Delta_i(f)=\frac{f-f|_{y_i=0}}{y_i}.
$$

\begin{thm}
\label{t:dde}
The ordinary generating function $G_k$ satisfies the following equation
$$
\left(\Delta_1\dots\Delta_k-(\Delta_1+\Delta_2)\dots
(\Delta_{k-1}+\Delta_k)\right)G_k=0.
$$
\end{thm}

It is known that the ordinary generating functions $G_k$ can be
obtained from exponential generating functions $E_k$ by the Laplace transform.
Thus Theorem \ref{t:dde} can in principle be deduced from Theorem
\ref{t:pde} and the properties of the Laplace transform.
However, we will give a direct proof.

For $k=1$, 2 and 3, the generating functions $G_k$ can be computed explicitly.
It is easy to see that
$$
G_1(y_1)=\frac 1{1-y_1},\quad G_2(y_1,y_2)=\frac 1{1-y_1-y_2}.
$$
We will prove the following theorem:

\begin{thm}
  \label{t:G3}
  The function $G_3(x,y,z)$ is equal to
$$
\frac{2xz-y(1-x-z)-y\sqrt{1-2(x+z)+(x-z)^2}}
{2(1-x-z)((x+y)(y+z)-y)}.
$$
\end{thm}

The numbers $V_{k,\ell,m}=V(1^k2^\ell 3^m)$ can be alternatively expressed
as coefficients of certain polynomials:

\begin{thm}
  \label{t:coeff}
  The number $V_{k,\ell,m}$ for $k>0$, $\ell>0$, $m>0$ is equal to the coefficient with
  $x^kz^m$ in the polynomial
  $$
  \frac{1-xz}{1+xz}\left((1+x)^{k+\ell+m}(1+z)^{k+\ell+m}-(x+z)^{k+\ell+m}\right).
  $$
\end{thm}

Set $s=k+\ell+m$.
Note that, since the term $(x+z)^s$ is homogeneous of degree $s$,
the number $V_{k,\ell,m}$, where $k,\ell, m>0$, is also equal to the coefficient with
$x^kz^m$ in the power series
$$
\frac{(1-xz)(1+x)^s(1+z)^s}{1+xz}.
$$
This implies the following explicit formula for the numbers $V_{k,\ell,m}$
($k,\ell,m >0$):
$$
V_{k,\ell,m}=\binom sk\binom sm+2\sum_{i=1}^k(-1)^i\binom s{k-i}\binom s{m-i}.
$$
Note that the sum $\sum_{i=1}^k(-1)^i\binom s{k-i}\binom s{m-i}$
can be expressed as the value of the generalized hypergeometric function $_3F_2$,
namely, it is equal to $\binom s{k-1}\binom s{m-1}\,
_3F_2(1,1-k,1-m;2+\ell+m,2+k+\ell;-1)$.

\section{Recurrence relations}

Let $R$ be the polynomial ring in countably many variables
$x_1$, $x_2$, $x_3$, $\dots$.
Define a linear operator $A:R\to R$ by the following formula:
$$
A(x_{i_1}x_{i_2}\dots x_{i_k}f)=
(x_{i_1}+x_{i_2})(x_{i_2}+x_{i_3})\dots (x_{i_{k-1}}+x_{i_k})f.
$$
In this formula, $x_{i_1}$, $\dots$, $x_{i_k}$ is any finite subset of
variables, and $f(x_{i_1},x_{i_2},\dots, x_{i_k})$ is any polynomial
in these variables.
The formula displayed above defines a linear action of $A$ on the entire $R$.
Indeed, any polynomial $P\in R$ can be represented as the sum
$P=\sum_{S} P_S$, where $S$ runs through all finite subsets of variables,
and $P_S$ denotes the sum of all terms (monomials together with coefficients)
that involve exactly all variables from $S$ and no other variables.
The polynomial $A(P_S)$ is defined above, and we extend $A$ by linearity.
We also set $A(1)=1$ by definition.

The operator $A$ thus defined reduces the degrees of all nonconstant
polynomials.
Therefore, for any polynomial $P$, there exists a positive integer $N$
such that $A^{N}(P)$ is a constant, which is independent of the
choice of $N$ provided that $N$ is sufficiently large.
We let $A^{\infty}(P)$ denote this constant.

\begin{prop}
\label{p:Ainf}
  We have
  $$
  V(1^{i_1}\dots k^{i_k})=A^\infty(x_1^{i_1}\dots x_k^{i_k}).
  $$
\end{prop}

\begin{proof}
We will argue by induction on the degree $i_1+\dots+i_k$, equivalently,
on the dimension of the Gelfand--Zetlin polytope $GZ(1^{i_1}\dots k^{i_k})$.
Let $\pi$ be the linear projection of $GZ(1^{i_1}\dots k^{i_k})$ to the cube $C$
given in coordinates $(u_1,\dots,u_{k-1})$ by the inequalities
$$
1\le u_1\le 2\le u_2\le\dots\le k-1\le u_{k-1}\le k.\eqno{(C)}
$$
Namely, we set $u_1=u_{1,i_1}$, $u_2=u_{1,i_1+i_2}$, $\dots$,
$u_{k-1}=u_{1,i_1+\dots+i_{k-1}}$.
Observe that all vertices of $GZ(p)$ project to vertices of the cube $C$.
Thus it suffices to describe the fibers of the projection $\pi$
over the vertices of the cube $C$.

It will be convenient to label the vertices of the cube $C$
by the monomials in the expansion of the polynomial
$A(x_1\dots x_k)$.
Namely, to fix a vertex of $C$, one needs to specify, for
every $j$ between 1 and $k-1$, which of the two inequalities
$j\le u_j$ or $u_j\le j+1$ turns to an equality.
Similarly, to fix a monomial in the polynomial $A(x_1\dots x_k)$,
one needs to specify, for every $j$ between 1 and $k-1$, which
term is taken from the factor $(x_j+x_{j+1})$, the term $x_j$ or
the term $x_{j+1}$.
This description makes the correspondence clear.

Let $v$ be the vertex of the cube $C$ corresponding to a
monomial $x_1^{\alpha_1}\dots x_k^{\alpha_k}$.
It is not hard to see that the polytope $\pi^{-1}(v)$ is
combinatorially equivalent to
$$
GZ(1^{i_1-1+\alpha_1}\dots k^{i_k-1+\alpha_k}).
$$

Suppose that
$$
A(x_1\dots x_k)=\sum_{\alpha_1,\dots,\alpha_k}
c_{\alpha_1\dots\alpha_k} x_1^{\alpha_1}\dots x_k^{\alpha_k}.
$$
Then we have
$$
V(1^{i_1}\dots k^{i_k})=\sum_{\alpha_1,\dots,\alpha_k}
c_{\alpha_1\dots\alpha_k} V(1^{i_1-1+\alpha_1}\dots k^{i_k-1+\alpha_k}).
$$
Since for any $k$-tuple of indices $\alpha_1$, $\dots$, $\alpha_k$,
for which the corresponding coefficient $c_{\alpha_1,\dots,\alpha_k}$
is nonzero, the Gelfand--Zetlin polytope
$GZ(1^{i_1-1+\alpha_1}\dots k^{i_k-1+\alpha_k})$ has smaller
dimension than $GZ(1^{i_1}\dots k^{i_k})$, we can assume by induction
that
$$
V(1^{i_1-1+\alpha_1}\dots k^{i_k-1+\alpha_k})=
A^\infty(x_1^{i_1-1+\alpha_1}\dots x_k^{i_k-1+\alpha_k}).
$$
Hence we have
$$
V(1^{i_1}\dots k^{i_k})=\sum_{\alpha_1,\dots,\alpha_k}
c_{\alpha_1\dots\alpha_k} A^\infty(x_1^{i_1-1+\alpha_1}\dots x_k^{i_k-1+\alpha_k})=
A^\infty(A(x_1^{i_1}\dots x_k^{i_k})).
$$
The desired statement follows.
\end{proof}

\section{Equations on generating functions $E_k$ and $G_k$}
In this section, we deduce equations on the generating functions
$E_k$ and $G_k$.
In particular, we prove Theorems \ref{t:pde} and \ref{t:dde}.

For a multi-index $\alpha=(\alpha_1,\dots,\alpha_k)$, we let
$z^\alpha$ denote the monomial $z_1^{\alpha_1}\dots z_k^{\alpha_k}$,
and $\alpha!$ denote the product $\alpha_1!\dots \alpha_k!$.
The partial derivation with respect to $z_\ell$ will be written as $\d_\ell$.
The power $\d^\alpha$ will mean $\d_1^{\alpha_1}\dots \d_k^{\alpha_k}$.
We will write $I_\ell$ for the operator of integration with
respect to the variable $z_\ell$.
This operator acts on the power series $\sum_{n=0}^\infty a_n z_\ell^n$,
where $a_n$ are power series in the other variables, as follows:
$$
I_\ell\left(\sum_{n=0}^\infty a_n z_\ell^n\right)=\sum_{n=0}^\infty
a_n\frac {z^{\ell+1}}{\ell+1}.
$$
We will use the expansion
$$
(x_1+x_2)\dots (x_{k-1}+x_k)=\sum_\alpha c_\alpha x^\alpha,
$$
in which the coefficients $c_\alpha$ can be computed explicitly.
Let $E^*_k$ be the sum of all terms in $E_k$ divisible by $z_1\dots z_k$.
Then we have ($i$, $j$, $\alpha$ being multi-indices of dimension $k$)
$$
E^*_k=\sum_{i>0} A^\infty(x^i)\frac {z^i}{i!}=
\sum_{i>0}\sum_\alpha c_\alpha A^\infty(x^{i-1+\alpha})\frac{z^i}{i!}=
$$
$$
=\sum_\alpha c_\alpha\d^{\alpha}I_1\dots I_k \sum_{i>0} A^\infty(x^{i-1+\alpha})
\frac{z^{i-1+\alpha}}{(i-1+\alpha)!}=
\sum_\alpha c_\alpha\d^{\alpha}I_1\dots I_k \sum_{j\ge\alpha} A^\infty(x^j)
\frac{z^j}{j!}.
$$
Apply the differential operator $\d_1\dots\d_k$ to both sides of this
equation.
Note that $\d_1\dots\d_k(E^*_k)=\d_1\dots\d_k(E_k)$.
Thus we have
$$
\d_1\dots\d_k(E_k)=\sum_\alpha c_\alpha\d^{\alpha}\sum_{j\ge\alpha} A^\infty(x^j)
\frac{z^j}{j!}.
$$
Observe also that, since $\alpha\ge 0$ whenever $c_\alpha\ne 0$, we have
$$
\d^{\alpha}\sum_{j\ge\alpha} A^\infty(x^j)
\frac{z^j}{j!}=\d^\alpha E_k.
$$
This implies Theorem \ref{t:pde}.

\medskip

\noindent{\sc Example: $k=1$ and $k=2$.}
In the case $k=1$, we have $E_1=e^{z_1}$.
Consider now the case $k=2$.
Set $E=E_2$, $x=z_1$ and $y=z_2$.
By Theorem \ref{t:pde}, the function $E$ satisfies the following partial
differential equation:
$$
E_{xy}=E_x+E_y.
$$
This equation can be simplified by setting $E=e^{x+y}u$,
then the function $u$ satisfies the equation
$$
u_{xy}=u.
$$
and the boundary value conditions $u(x,0)=u(0,y)=1$.
We can now look for solutions $u$ that have the form $v(xy)$, where
$v$ is some smooth function.
This function must satisfy the initial condition $v(0)=1$
and the ordinary differential equation
$$
tv''(t)+v'(t)-v(t)=0.
$$
It is known that the only analytic solution of this initial
value problem is $I_0(2\sqrt t)$, where $I_0$ is the modified Bessel function
of the first kind.
Thus $I_0(2\sqrt{xy})$ is a partial solution of the boundary value problem
$u_{xy}=u$, $u(x,0)=u(0,y)=1$.
The solution of this boundary value problem is unique (note that the
boundary values are defined on characteristic curves!).
Therefore, we must conclude that $E(x,y)=e^{x+y}I_0(2\sqrt{xy})$.

The proof of Theorem \ref{t:dde} is very similar to the proof of
Theorem \ref{t:pde}.
Let $G_k^*$ be the sum of all terms in $G_k$ that are divisible by
$y_1\dots y_k$, i.e.
$$
G^*_k=\sum_{i>0} V(1^{i_1}\dots k^{i_k}) y^i.
$$
Then, similarly to a formula obtained for $E^*_k$, we have
$$
G^*_k=\sum_\alpha c_\alpha y_1^{1-\alpha_1}\dots y_k^{1-\alpha_k}
\sum_{j\ge\alpha} A^\infty(x^j)y^j.
$$
Applying the operator $\Delta_1\dots\Delta_k$ to both sides of this
equation, we obtain Theorem \ref{t:dde}.
Similarly to the proof of Theorem \ref{t:pde}, we need to use
that
$$
\Delta_1\dots\Delta_k(G^*_k)=\Delta_1\dots\Delta_k(G_k)
$$
and that
$$
\Delta_1^{\alpha_1}\dots\Delta_k^{\alpha_k}(G_k)=
y_1^{-\alpha_1}\dots y_k^{-\alpha_k}\sum_{j\ge\alpha} A^\infty(x^j)y^j.
$$

We will now discuss several examples.

\medskip

\noindent{\sc Example: $k=1$ and $k=2$.}
For $k=1$, we have the following equation: $\Delta_1 G_1=G_1$, i.e.
$G_1(y_1)-G_1(0)=y_1G_1(y_1)$.
Knowing that $G_1(0)=1$, this gives
$$
G_1(y_1)=1+y_1+y_1^2+\dots=\frac 1{1-y_1}.
$$
Suppose that $k=2$.
Set $G=G_2$, $x=y_1$, $y=y_2$.
The function $G$ satisfies the following equation
$$
\Delta_x\Delta_y G=\Delta_x G+\Delta_y G.
$$
Note that $G(x,0)=G_1(x)$ and $G(0,y)=G_1(y)$.
Therefore, the right-hand side can be rewritten as
$$
\frac{G-\frac 1{1-y}}x+\frac{G-\frac 1{1-x}}y.
$$
The left-hand side is
$$
\Delta_x\left(\frac{G-\frac 1{1-x}}{y}\right)=
\frac 1x\left(\frac{G-\frac 1{1-x}}{y}-\frac{\frac 1{1-y}-1}y\right).
$$
Solving the linear equation on $G$ thus obtained, we conclude that
$$
G=\frac 1{1-x-y}.
$$

\medskip

\noindent{\sc Example: $k=3$.}
We set $G=G_3$, $x=y_1$, $y=y_2$ and $z=y_3$.
The function $G$ satisfies the following equation:
$\Delta_x\Delta_y\Delta_zG=(\Delta_x+\Delta_y)(\Delta_y+\Delta_z)G$.
This equation can be rewritten as follows:
$$
\Delta_y^2G=\frac{G(1-x-y-z)-1}{xyz}.
$$
Suppose that $G=G(x,0,z)+A(x,z)y+\dots$, where dots denote the terms divisible
by $y^2$.
Then we have
$$
\Delta_y^2G=G-G(x,0,z)-A(x,z)y=G-\frac 1{1-x-z}-A(x,z)y.
$$
Substituting this into the equation, we can solve the equation for $G$
in terms of $A$:
$$
G=\frac{-x z+y (1-x-z) (1-A(x,z) x z)}{(1-x-z) \left(y-(x+y)(y+z)\right)}.
$$
Since the power series $1-x-z$ is invertible, it follows that $G$ has the form
$$
\frac{a+by}{y-(x+y)(y+z)},
$$
where $a$ and $b$ are some power series in $x$ and $z$.
Let $\lambda$ and $\mu$ be the two solutions of the equation
$y=(x+y)(y+z)$, namely,
$$
\lambda,\mu=\frac{1-x-z\pm\sqrt{1-2(x+z)+(x-z)^2}}2.
$$
The signs are chosen so that, at the point $x=z=0$, we have $\lambda=1$
and $\mu=0$.
Then
$$
\frac 1{y-(x+y)(y+z))}=\frac{c}{y-\lambda}+\frac d{y-\mu},
$$
where $c$ and $d$ are some power series in $x$ and $z$.
Note that, since $(y-\lambda)^{-1}$ makes sense as a power series,
$c(a+by)/(y-\lambda)$ can be represented as a power series in $x$, $y$ and $z$.
Thus the function $d(a+by)/(y-\mu)$ must also be representable as
a power series in $x$, $y$ and $z$.
However, this is only possible if the numerator is a multiple of the
denominator, i.e. $(a+by)=e(y-\mu)$, where the
coefficient $e$ is a power series of $x$ and $z$.
It follows that $G$ is equal to $e(y-\lambda)^{-1}$.
The coefficient $e$ can be found from the condition $G(x,0,z)=\frac 1{1-x-z}$:
$$
G=\frac 1{1-x-z}\frac{\lambda}{\lambda-y}=
$$
$$
=\frac{2xz-y(1-x-z)-y\sqrt{1-2(x+z)+(x-z)^2}}
{2(1-x-z)((x+y)(y+z)-y)}.
$$

\section{Proof of Theorem \ref{t:coeff}}
\label{s:gz3}
In this section, we will prove Theorem \ref{t:coeff}, which expresses
the numbers $V_{k,\ell,m}$ as coefficients of certain Laurent polynomials.
The numbers $V_{k,\ell,m}$ satisfy the following recurrence relation:
$$
V_{k,\ell,m}=V_{k-1,\ell,m}+V_{k,\ell-1,m}+V_{k,\ell,m-1}+V_{k-1,\ell+1,m-1}
$$
provided that $k$, $\ell$, $m>0$,
and the following initial conditions:
$$
V_{0,\ell,m}=V_{\ell,m},\quad V_{k,0,m}=V_{k,m},\quad V_{k,\ell,0}=V_{k,\ell}.
$$

Set $V_{k,m}^s=V_{k,s-k-m,m}$.
Then we can write the following recurrence relations on the numbers $V_{k,m}^s$:
$$
V_{k,m}^s=V_{k-1,m}^{s-1}+V_{k,m-1}^{s-1}+V_{k-1,m-1}^{s-1}+V_{k,m}^{s-1}
$$
provided that $k\ge 1$, $m\ge 1$, $k+m\le s-1$, and
$$
V_{k,m}^s=V_{k-1,m}^{s-1}+V_{k,m-1}^{s-1}+V_{k-1,m-1}^{s-1}
$$
provided that $k+m=s$.

\begin{figure}
  \includegraphics[height=1cm]{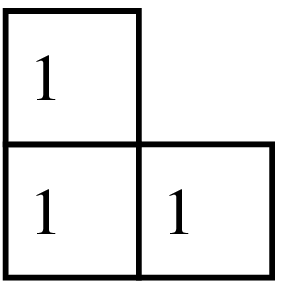}
  \includegraphics[height=1.5cm]{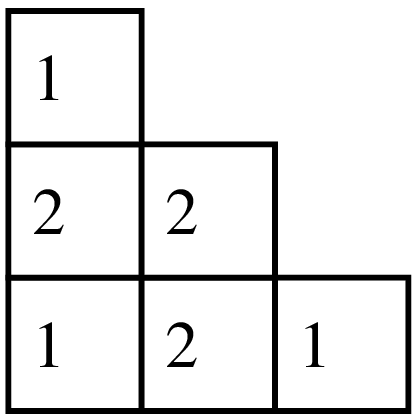}
  \includegraphics[height=2cm]{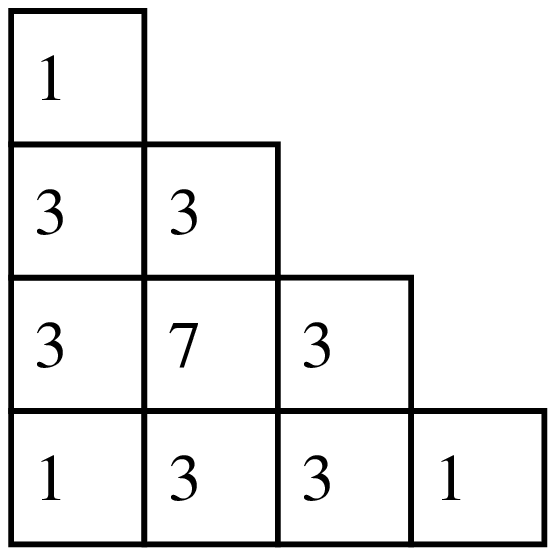}
  \includegraphics[height=2.5cm]{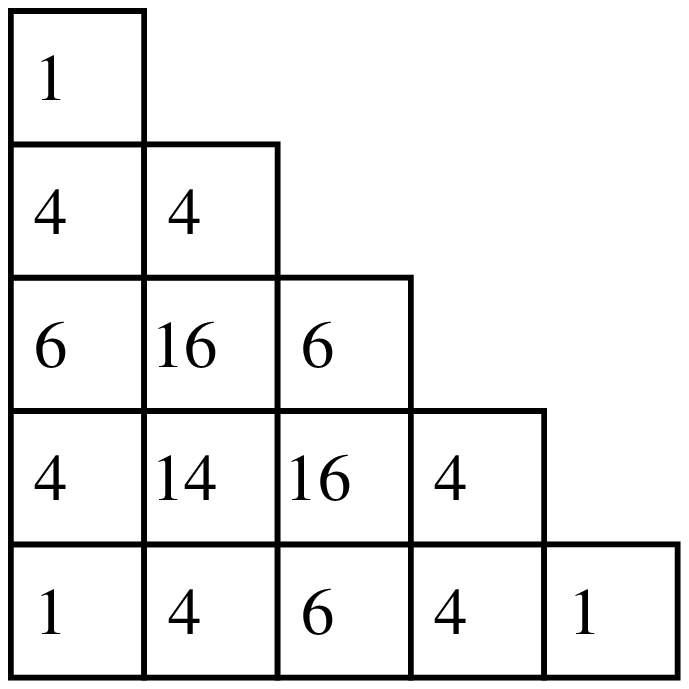}
  \includegraphics[height=3cm]{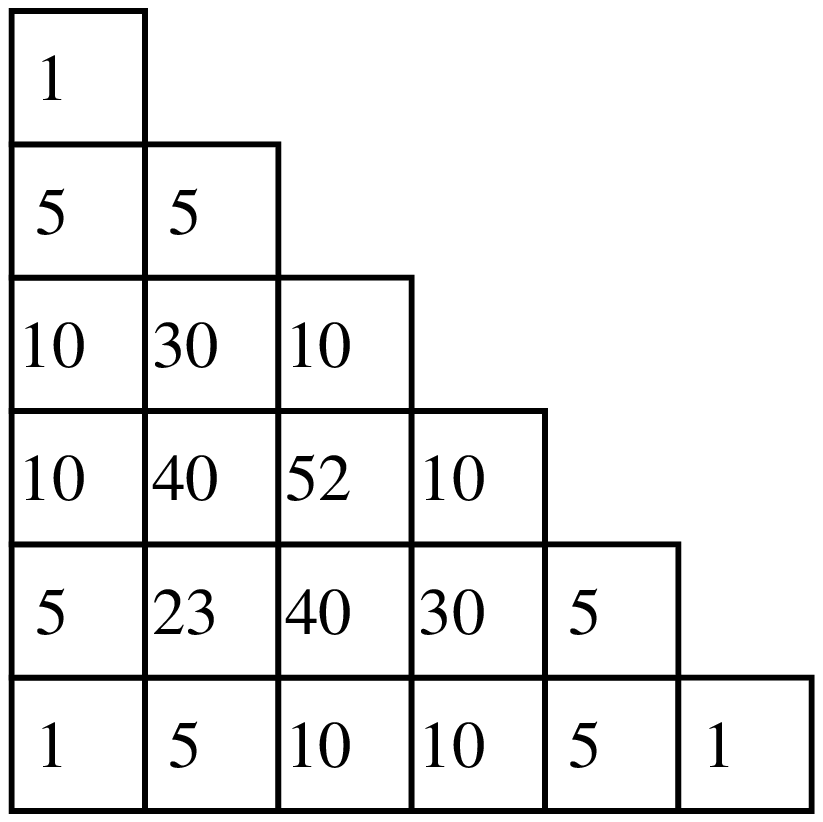}
  \caption{Triangular tables $T^s$ containing the numbers $V^s_{k,m}$.
Southwest corners of these tables are located at $(0,0)$.}
  \label{pic:T}
\end{figure}

For a fixed $s$, we can arrange the numbers $V^s_{k,m}$ into a triangular
table $T^s$ of size $s$ as shown on Figure \ref{pic:T}.
Namely, the number $V^s_{k,m}$ is placed into the cell, whose southwest (lower left)
corner is at position $(k,m)$.
The next table $T^{s+1}$ can be obtained from the table $T^s$ as follows.
First, we add to every element of $T^s$ its south, west and southwest neighbors.
Next, we add a line of cells, whose positions $(k,m)$ satisfy the equality $k+m=s$.
In every cell of this line, we put the sum of the south and west neighbors.
Note that, by construction, the boundary of every table $T^s$ consists of
binomial coefficients.

Consider the generating function $G=G_3$ for the numbers $V_{k,\ell,m}$.
The splitting of $G$ into homogeneous components can be
obtained by expanding the function $G(xy,y,zy)$ into powers of $y$.
We set
$$
G(xy,y,zy)=\sum_{s=0}^\infty g_s(x,z)y^s
$$
Then we have
$$
g_s(x,z)=\sum_{k=0}^s\sum_{m=0}^{s-k} V^s_{k,m} x^kz^m.
$$
Thus the coefficients of the polynomial $g_s$ are precisely elements of
the table $T^s$.
The recurrence relations on the numbers $V^s_{k,m}$ displayed above
imply the following property of the generating functions $g_s$:

\begin{prop}
  The polynomials $g_s$ satisfy the following recurrence relations:
  $$
  g_{s+1}=(1+x+z)g_s+\tau_{\le s}(xz g_{s}),
  $$
  where the truncation operator $\tau_{\le s}$ acts on a polynomial
  by removing all terms, whose degrees exceed $s$.
\end{prop}

Consider the polynomials
$$
h_s(x,z)=g_s(x,z)-(xz)^s g_s(z^{-1},x^{-1}).
$$
Geometrically, these polynomials can be described as follows.
Let $\tilde T^{s}$ denote the table, into which we put all coefficients of the
polynomial $h_s$, see Figure \ref{pic:Ttilde}.
The lower left triangle of size $s-1$ is the same in the tables $T^s$
and $\tilde T^{s}$.
The table $\tilde T^{s}$ is skew-symmetric with respect to the main diagonal.
These two properties give a unique characterization of the tables $\tilde T^{s}$.

\begin{figure}
  \includegraphics[height=1cm]{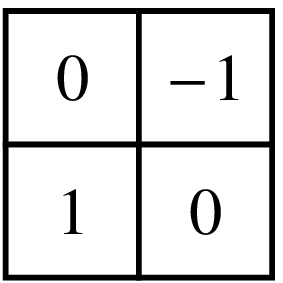}
  \includegraphics[height=1.5cm]{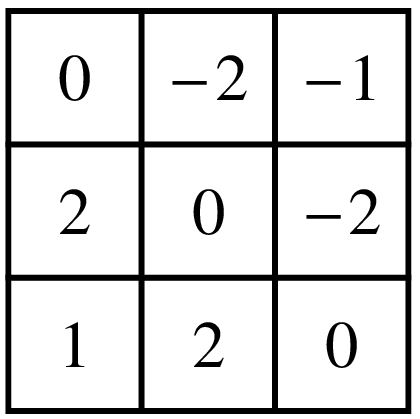}
  \includegraphics[height=2cm]{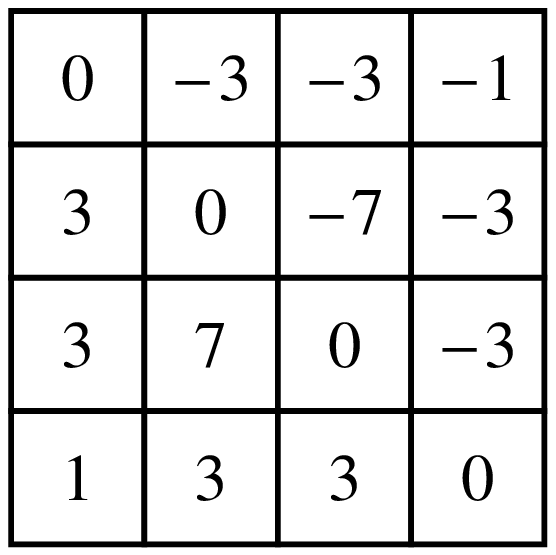}
  \includegraphics[height=2.5cm]{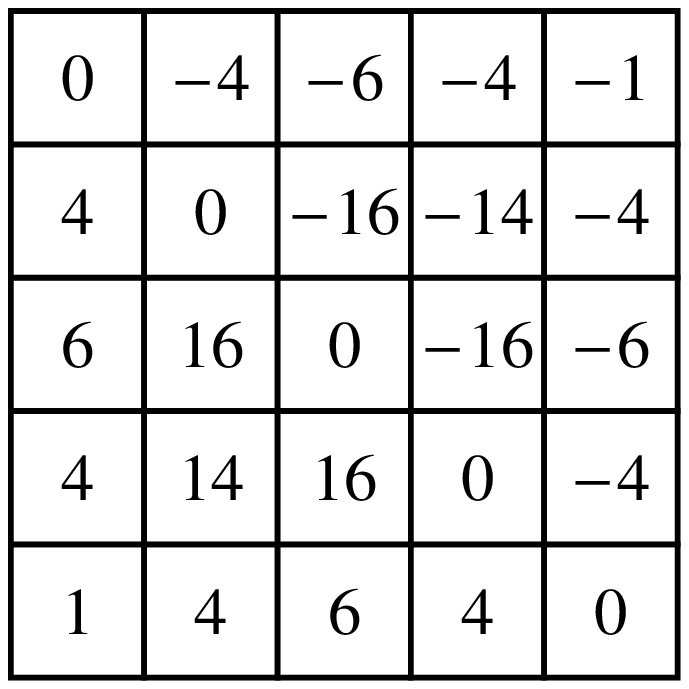}
  \includegraphics[height=3cm]{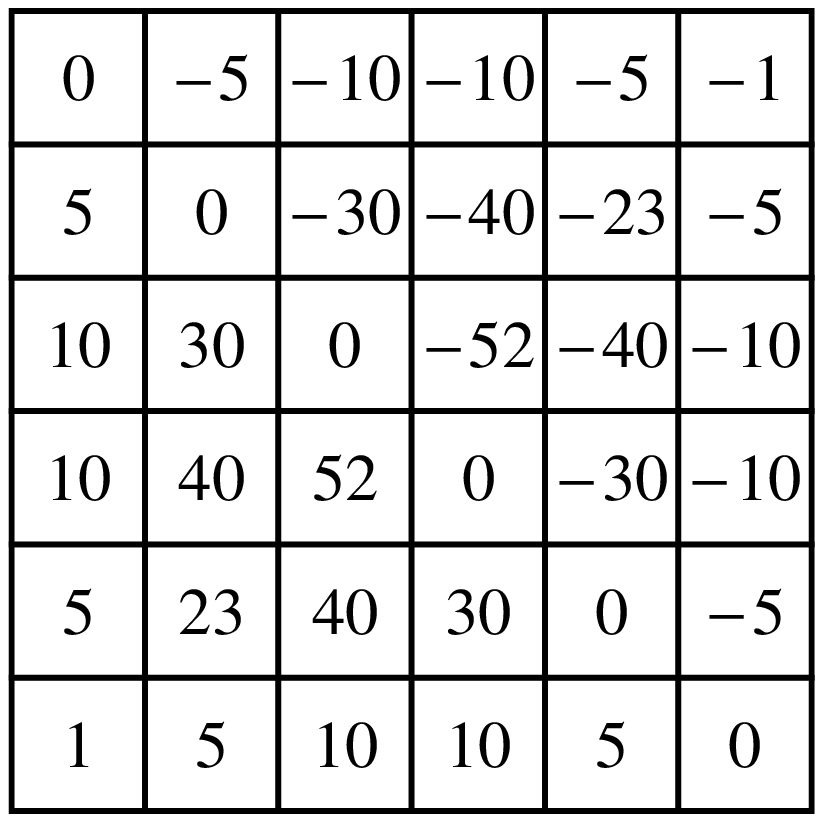}
  \caption{The skew-symmetric tables $\tilde T^{s}$.}
  \label{pic:Ttilde}
\end{figure}

\begin{figure}
  \includegraphics[height=2cm]{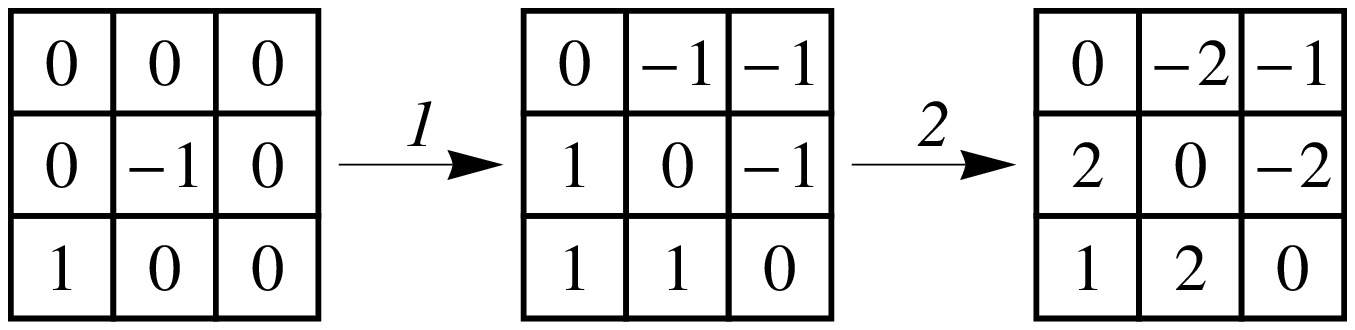}\\
  \includegraphics[height=2.5cm]{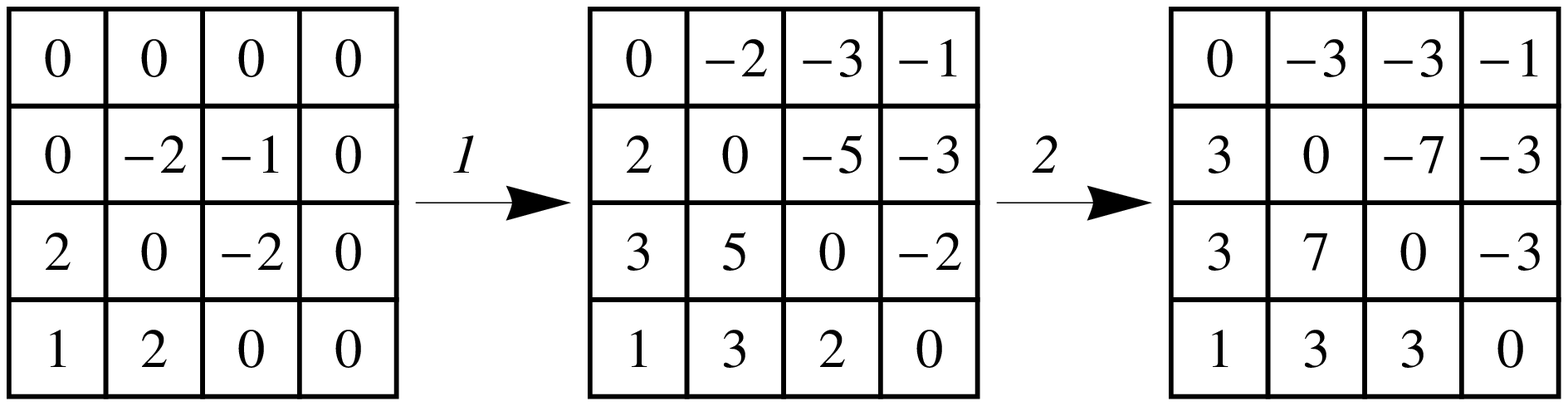}
  \caption{The rules of generating the tables $\tilde T^{s}$.}
  \label{pic:rule}
\end{figure}

The rules, by which the tables $\tilde T^s$ are formed, are the following
(see Figure \ref{pic:rule}).
The first table $\tilde T^1$ is by definition the left-most table shown on Figure \ref{pic:Ttilde}.
The next table $\tilde T^{s+1}$ is obtained inductively from the preceding table $\tilde T^s$
in two steps.
In the first step, we add to every element of $\tilde T^s$ its immediate west, south and southwest
neighbors.
In the second step, we modify elements in two diagonals of the table, namely, the elements,
whose positions (mesured by southwest corners) $(k,m)$ satisfy the equality $k+m=s$ or $k+m=s+2$.
To the cell at position $(k,m)$, where $k+m=s$, we add the binomial coefficient $\binom{k+m}{m}$.
From the cell at position $(k+1,m+1)$, we subtract this binomial coefficient.

We have the following recurrence relation on the polynomials $h_s$:
$$
h_{s+1}=h_s(1+x)(1+z)+(1-xz)(x+z)^s,
$$
which does not contain truncation operators.
Therefore, the generating function $H=\sum_{s=0}^\infty h_s y^s$
satisfies the following linear equation:
$$
H=1+y((1+x)(1+z)H+(1-xz)(1-y(x+z))^{-1}).
$$
Solving this equation, we find that
$$
H=\frac{y(1-xz)}{(1-y(x+z))(1-y(1+x)(1+z))}.
$$

Knowing the generating function $H$, we can now obtain an explicit formula
for the polynomials $h_s$, namely,
$$
h_s(x,z)=\frac{1-xz}{1+x z}\left((1+x)^s(1+z)^s-(x+z)^s\right).
$$
Theorem \ref{t:coeff} is thus proved.

\subsection*{Open problems}
\begin{enumerate}
 \item Prove or disprove: the generating function $G_4$ is algebraic.
Note that $G_1$ and $G_2$ are rational, and $G_3$ is algebraic.
\item Deduce differential or difference equations on the generating functions
for the $f$-vectors and for the modified $h$-vectors of Gelfand--Zetlin polytopes.
\end{enumerate}

\end{document}